\documentclass[11pt]{amsart}

\usepackage[T1]{fontenc}
\usepackage[utf8]{inputenc}
\usepackage{lmodern}
\usepackage{stmaryrd}

\usepackage[american]{babel}
\usepackage[babel, final]{microtype}
\usepackage[all]{xy}
\usepackage{ifpdf}
\usepackage{comment}
\usepackage{multirow}

\usepackage[pdftex, dvipsnames]{xcolor}
\usepackage[pdftex, final]{graphicx}
\usepackage{pinlabel}
\usepackage[pdftex,%
    includehead,%
    includefoot,%
    nomarginpar,%
    lmargin=1in,%
    rmargin=1in,%
    tmargin=1in,%
    bmargin=1in,%
]{geometry}
\usepackage{amsmath,amsthm,amssymb,amsfonts,mathrsfs,tikz,tikz-cd,bm,verbatim,mathtools,hyperref,caption,enumitem,csquotes,float}

\hypersetup{
    colorlinks=true,
    linkcolor=NavyBlue,
    citecolor=NavyBlue
}

\makeatletter
\def\@tocline#1#2#3#4#5#6#7{\relax
  \ifnum #1>\c@tocdepth 
  \else
    \par \addpenalty\@secpenalty\addvspace{#2}%
    \begingroup \hyphenpenalty\@M
    \@ifempty{#4}{%
      \@tempdima\csname r@tocindent\number#1\endcsname\relax
    }{%
      \@tempdima#4\relax
    }%
    \parindent\z@ \leftskip#3\relax \advance\leftskip\@tempdima\relax
    \rightskip\@pnumwidth plus4em \parfillskip-\@pnumwidth
    #5\leavevmode\hskip-\@tempdima
      \ifcase #1
       \or\or \hskip 1em \or \hskip 2em \else \hskip 3em \fi%
      #6\nobreak\relax
    \dotfill\hbox to\@pnumwidth{\@tocpagenum{#7}}\par
    \nobreak
    \endgroup
  \fi}
\makeatother

\usepackage[hang,flushmargin]{footmisc}

\usepackage[backend=biber,style=alphabetic,sorting=nyt, url=false,doi=false,isbn=false,backref=true]{biblatex}
\DefineBibliographyStrings{english}{
  backrefpage={$\uparrow$},
  backrefpages={$\uparrow$}
}

\setlist[itemize]{noitemsep} 

\addto\extrasamerican{%
}

\newtheorem{theorem}{Theorem}[section]
\newtheorem{corollary}{Corollary}[section]
\newtheorem{proposition}{Proposition}[section]
\newtheorem{lemma}{Lemma}[section]
\theoremstyle{definition}
\newtheorem{definition}{Definition}[section]
\newtheorem{example}{Example}[section]
\newtheorem{remark}{Remark}[section]

\newtheorem{conjecture}{Conjecture}[section]

\makeatletter
\let\c@conjecture=\c@theorem
\let\c@corollary=\c@theorem
\let\c@proposition=\c@theorem
\let\c@lemma=\c@theorem
\let\c@definition=\c@theorem
\let\c@example=\c@theorem
\let\c@remark=\c@theorem
\let\c@equation\c@theorem
\let\c@question\c@theorem
\makeatother

\def\makeautorefname#1#2{\expandafter\def\csname#1autorefname\endcsname{#2}}
\makeautorefname{proposition}{Proposition}
\makeautorefname{lemma}{Lemma}
\makeautorefname{definition}{Definition}
\makeautorefname{example}{Example}
\makeautorefname{question}{Question}
\makeautorefname{conjecture}{Conjecture}
\makeautorefname{section}{Section}
\makeautorefname{claim}{Claim}
\makeautorefname{equation}{Equation}
\makeautorefname{remark}{Remark}

\newcommand{\ZZ}{\mathbb{Z}}

\newcommand{\RR}{\mathbb{R}}

\title{The Wrappingness and Trunkenness of Volume-Preserving Flows}
\date{March 2024}

\author{Peter Lambert-Cole}
\address{Department of Mathematics, University of Georgia, Athens, GA 30602}
\email{\href{plc@uga.edu}{plc@uga.edu}}

\begin{document}

\maketitle

\begin{abstract}
    Link invariants of long pieces of orbits of a volume-preserving flow can be used to define diffeomorphism invariants of the flow.  In this paper, we extend the notions of wrapping number and trunk and define invariants of links with respect to a fibration on a 3-manifold.  Extending work of Dehornoy and Rechtman, we apply this to define diffeomorphism invariants {\it wrappingness} and {\it trunkenness} of volume-preserving flows on 3-manifolds and interpret these invariants as obstructions to the existence of a global surface of section for the flow.  Finally, we construct flows and show that wrappingness and trunkenness are not functions of the helicity of a flow.
\end{abstract}

\section{Introduction}

\subsection{Motivation}

The motivation for this work is understanding surfaces of sections for volume-preserving flows on 3-manifolds.

\begin{definition}
    A {\it global surface of section} for a flow $\phi^t$ is a compact, embedded surface $\Sigma \subset Y$ such that
    \begin{enumerate}
        \item the flow $\phi^t$ is transverse to the interior of $\Sigma$,
        \item the boundary $\partial \Sigma$ is union of periodic orbits of the flow $\phi^t$, and
        \item for each point $p \in Y \setminus \partial \Sigma$, there exist $t_- < 0 < t_+$ such that $\phi^{t_-}(p)$ and $\phi^{t_+}(p)$ lie in the interior of $\Sigma$.
    \end{enumerate}
\end{definition}

One specific inspiration for this work is the following problem.  Let $(X,\omega)$ be a compact symplectic 4-manifold and $Y \subset X$ a hypersurface.  The hypersurface is {\it contact-type} if locally there exists a primitive $\alpha$ for $\omega$ such that $\alpha \wedge d \alpha$ is a volume-form on $Y$.  In this case, the form $\alpha$ is a contact form for some contact structure on $Y$.  More generally, one can ask if there exists a contact structure $\xi$ on $Y$ such that $\omega|_{\xi} > 0$, in which case we say that $\omega$ {\it dominates} $\xi$.  If $Y = \partial X$, then $X$ is a {\it strong filling} of $(Y,\xi)$ if $Y$ is contact-type and a {\it weak filling} if $\omega$ dominates $\xi$.

For a volume form $dvol_Y$, the restriction $\omega|_Y$ determines a volume-preserving flow $\Phi_R$ defined by integrating the unique vector field $R$ satisfying
\[dvol_Y(R,-) = \omega|_Y\]
The question of whether $Y$ is contact-type or admits a contact structure dominated by $\omega$ can be interpreted in terms of the dynamics of the flow $\Phi_R$.  Let $(Y,\xi)$ be a contact structure and $(B,\pi)$ an open book decomposition supporting this contact structure.  Suppose that $B$ is a union of periodic orbits of the flow $\Phi_R$.  A natural question is whether there exists an ambient isotopy of the fibration $\pi: Y \setminus B \rightarrow S^1$ until the flow of $\Phi_R$ is transverse to the fibration -- i.e. is $B$ the binding of a global surface of section for the flow of $\Phi_R$?  In this case, the Hamiltonian structure $\omega|_Y$ dominates the contact structure $(Y,\xi)$ and if $Y = \partial X$, then $(X,\omega)$ is a weak symplectic filling of $(Y,\xi)$.  Moreover, there exists an extension of $\omega$ onto a collar neighborhood $[0,1] \times Y$ of $\partial X$ such that $(X,\omega)$ is a strong symplectic filling of $(Y,\xi)$.  

The existence of a global surface of section is a diffeomorphism invariant of the dynamical system determined by the Hamiltonian structure $\omega|_Y$.  Let $\pi: Y \rightarrow S^1$ be a fibered 3-manifold (possibly with toroidal boundary obtained by removing tubular neighborhoods of some periodic orbits).  A necessary condition for the surface $\pi^{-1}(\theta)$ to be isotopic to a global surface of section is that each periodic orbit is, up to isotopy, braided with respect to this fibration.  More generally, any link invariant that measures the obstruction to being braided can be applied to long pieces of orbits in order to obstruct the existence of a global surface of section.  Generalizing Dehornoy and Rechtman \cite{Dehornoy-Rechtman}, we define the {\it wrappingness} and {\it trunkenness} of a vector field with respect to a fibration.  We interpret these invariants as measuring an obstruction to finding a global surface of section to the flow. 

\subsection{Trunk and wrapping number}

The {\it trunk} of a link $L$ in $S^3$ was defined by Ozawa \cite{Ozawa}.  It is the minimum, over all Morse functions $f$ on $S^3$ with exactly two critical points, of the maximal geometric intersection number of $L$ with a regular level set of $f$.  Note that a knot $K \subset S^3$ is trivial if and only if its trunk is 2.  

Let $L$ be a satellite knot in $S^3$ with pattern $P$ and companion $K$.  This means that $P$ is a knot in the (framed) solid torus $S^1 \times D^2$ and the knot $L$ is obtained by taking $S^3 \setminus \nu(K)$ and gluing in $(S^1 \times D^2,P)$ to recover a knot in $S^3$.  The {\it wrapping number} of the pattern $P$ is the minimal geometric intersection number of $P$ with a meridonal disk of $S^1 \times D^2$ (up to ambient isotopy).

We generalize these two notions as follows:

\begin{definition}
Let $\pi: Y \rightarrow S^1$ be a fibered 3-manifold (possibly with boundary) and let $\Pi$ be the set of all smooth fibrations $\pi':Y \rightarrow S^1$ isotopic to $\pi$.  Let $L$ be a link in $Y$.
\begin{enumerate}
\item the {\it winding number} of $L$ with respect to the fibration is the algebraic intersection number of $L$ with a fiber.
\item the {\it wrapping number} of $L$ with respect to the fibration is the minimum geometric intersection number of $L$ with a fiber surface,
\item the {\it trunk} of $L$ with respect to the fibration is the minimum, over all smooth fibrations $\pi$ in $\Pi$, of the maximum geometric intersection number of $L$ with a leaf $\pi^{-1}(\theta)$.
\end{enumerate}
\end{definition}

\begin{definition}
If $K \subset Y$ is a fibered link and $L$ is a link in the complement of $Y$, the {\it winding number of $L$ with respect to $K$}, the {\it wrapping number of $L$ with respect to $K$}, and the {\it trunk of $L$ with respect to $K$} are defined in terms of the fibration $\pi: Y \setminus K \rightarrow S^1$. 

A link $L$ in $Y$ is {\it semifibered} if it admits a fibered sublink $L_1$ such that the complementary sublink $L_2 = L \smallsetminus L_1$ has wrapping number 0 with respect to the fibration.
\end{definition}

The fibration $\pi$ determines a class $\pi^*(d \theta) \subset H^1(Y;\ZZ)$.  The linking number can be computed by taking $\pi^*(d \theta)$, restricting to $L$, then pairing with the fundamental class in $H_1(L;\ZZ)$.  The link $L \subset Y$ is {\it braided} with respect to the fibration $\pi: Y \rightarrow S^1$ if the restriction $\pi: L \rightarrow S^1$ is a covering map.  The following statements follow immediately from the definitions.

\begin{proposition}
Let $\pi: Y \rightarrow S^1$ be a fibered 3-manifold (possibly with boundary) and $L \subset Y$ a link.  Then
\begin{enumerate}
    \item There is a sequence of inequalities
    \[wind(L,\pi) \leq wrap(L,\pi) \leq trunk(L,\pi)\]
    \item the trunk of $L$ with respect to $\pi$ is equal to the winding number if and only if $L$ is ambient isotopic to a link that is braided with respect to the fibration $\pi$.
    \item we have
    \[wrap(L,\pi) = 0 \qquad \text{ and } \qquad trunk(L,\pi) = 2\]
    if and only if $L$ can be isotoped to lie in a page of the fibration.
\end{enumerate}
\end{proposition}
In particular, the trunk of $L$ with respect to $\pi$ is an obstruction to $L$ being braided with respect to $\pi$.

The trunk of a disjoint of two links $L_1,L_2$ in $S^3$ satisfies the formula
\begin{equation}
\label{eq:trunk-additive}
   trunk(L_1 \cup L_2,S^3) = \text{max}(trunk(L_1,S^3),trunk(L_2,S^3) 
\end{equation}

Let $\pi_1: Y_1 \rightarrow S^1$ and $\pi_2: Y_2 \rightarrow S^2$ be fibered 3-manifolds with boundary and let $\pi: Y_1 \cup_{T^2} Y_2 \rightarrow S^1$ be their union along some fibered $T^2$-boundary component.  We conjecture the following generalization of Equation \ref{eq:trunk-additive}. 

\begin{conjecture}
\label{conj:fiber-sum}
Let $L_1 \subset Y_1$ and $L_2 \subset Y_2$ be two links and let $L_1 \cup L_2$ be their disjoint union in $Y_1 \cup Y_2$.  Then
\begin{enumerate}
    \item $wrap(L_1 \cup L_2,\pi) = wrap(L_1,\pi_1) + wrap(L_2,\pi_2)$,
    \item $trunk(L_1 \cup L_2,\pi) = \text{max}\left( trunk(L_1,\pi_1) + wrap(L_2,\pi_2), wrap(L_1,\pi_1) + trunk(L_2,\pi_2) \right)$
\end{enumerate}
\end{conjecture}

\subsection{Wrappingness and Trunkenness of vector fields}

Dehornoy and Rechtman used the trunk to define an asymptotic invariant of volume-preserving flows on $S^3$.  If $X$ preserves a volume form $\mu$, then Poincare recurrence implies that $\mu$-almost every point is recurrent.  In particular, $\mu$-almost every flowline can be approximated by a sequence of knots; therefore invariants of the knots, such as signature, trunk or any Vassiliev invariant, give invariants of the flow.  While most of these invariants are simply functions of the helicity of $X$, the trunkenness of a flow is not determined by its helicity.

Analogous to the generalization of trunk and wrapping number of knots with respect to a fibration, we generalize the trunkneness of a vector field with respect to a fibration.  First, note that if the invariant measure $\mu$ arises from a smooth volume form $\Omega$, then
\begin{align*}
    [\pi] \cdot \mu &= \int_{\pi^{-1}(\theta)} \iota_X \Omega & Wrp(X,\mu,\pi) &= \underset{\pi \in \Pi}{\text{inf}} \, \, \underset{\theta \in S^1}{\text{min}} \int_{\pi^{-1}(\theta)} | \iota_X \Omega | \\
    && Tks(X,\mu,\pi) &= \underset{\pi \in \Pi}{\text{inf}} \, \, \underset{\theta \in S^1}{\text{max}} \int_{\pi^{-1}(\theta)} | \iota_X \Omega | 
\end{align*}
Moreover, since $\Omega$ is $X$-invariant, we have that
\[0 = \mathcal{L}_X \Omega = \iota_X d\Omega + d \iota_X \Omega = d \iota_X \Omega = 0\]
so that $\iota_X \Omega$ is closed and the linking number is equal to the $\iota_X \Omega$-area of a page $\pi^{-1}(\theta)$.  Since all such fiber surfaces are homologous, this integral is independent of the fibration $\pi$ within its isotopy class.

Let $X_L$ be a nonsingular vector field tangent to the link $L$ and let $\phi^t_L$ be the flow of $X_L$.  There is a Dirac linear measure associated to $X_L$.  Given a measurable set $A \subset Y$ and $x_1,\dots,x_n$ an arbitrary collection of point on $L$, one on each component.  Then
\[\mu_L(A) := \sum_{i = 1}^n \text{Leb}\left( t \in [0,T_i] : \phi^t_L(x_i) \in A \right) \]
This measure is $X_L$-invariant and has total mass $T_L = T_1 + \cdots + T_n$.  If $p \in L \pitchfork S$ is a transverse intersection point of $L$ with a surface $S$, then the set $\mu^{[0,\epsilon]}(p)$ has $\mu_L$-measure $\epsilon$.  Therefore, the geometric intersection number of $L$ with $S$ is given by the formula
\[\#(L \pitchfork S) = \underset{\epsilon \rightarrow 0}{\lim} \, \, \frac{1}{\epsilon} \mu \left(\phi^{[0,\epsilon]}(L \cap S)\right) = \text{Flux}(X_L,\mu_L,S)\]
since $\mu_L$ is concentrated on $L$.

\begin{definition}
Let $X$ be a vector field on $Y^3$, let $\phi_X$ denote the flow of $X$, and let $\mu$ be an $X$-invariant Borel probability measure on $Y$.  Let $\pi: Y \rightarrow S^1$ be a smooth fibration and $\Pi$ the set of all fibrations isotopic to $\pi$.

\begin{enumerate}
\item The {\it Winding number} of $\mu$ with respect to a fibration is the integral
\[[\beta] \cdot \mu = \int_Y \beta(X) d\mu\]
where $\beta$ is a closed 1-form representing $[\pi] \subset H^1(Y,\ZZ) \subset H^1(Y,\RR)$.

\item The {\it Wrappingness} of $\mu$ with respect to a fibration is
\[Wrp(X,\mu,\pi) = \underset{\pi \in \Pi}{\text{inf}} \, \, \underset{\theta \in S^1}{\text{min}} \, \, \text{Flux}(X,\mu,\pi^{-1}(\theta)) = \underset{\pi \in \Pi}{\text{inf}} \, \, \underset{\theta \in S^1}{\text{min}} \, \, \underset{\epsilon \rightarrow 0}{\text{lim}} \, \, \frac{1}{\epsilon}\mu \left(\phi_X^{[0,\epsilon]}(\pi^{-1}(\theta)\right)\]

\item The {\it Trunkenness} of $\mu$ with respect to the fibration $\pi$ is
\[Tks(X,\mu,\pi) = \underset{\pi \in \Pi}{\text{inf}} \, \, \underset{\theta \in S^1}{\text{max}} \, \, \text{Flux}(X,\mu,\pi^{-1}(\theta)) = \underset{\pi \in \Pi}{\text{inf}} \, \, \underset{\theta \in S^1}{\text{max}} \, \, \underset{\epsilon \rightarrow 0}{\text{lim}} \, \, \frac{1}{\epsilon}\mu \left(\phi_X^{[0,\epsilon]}(\pi^{-1}(\theta)\right)\]
\end{enumerate}
\end{definition}

\begin{definition}
Let $X$ be a vector field on $Y$, let $\Phi_X$ denote the flow of $X$.  Let $B$ be a fibered link that is the union of periodic orbits of $\Phi_X$ and let $\pi: Y \setminus B \rightarrow S^1$ be the fibration.  Suppose that $\rho^{\pi}(B) > 0$ on each component of the binding.  Let $\mu$ be an $X$-invariant Borel probability measure on $Y \setminus B$. 

\begin{enumerate}
\item The {\it Winding number} of $\mu$ with respect to $B$ is the integral
\[[\beta] \cdot \mu = \int_Y \beta(X) d\mu\]
where $\beta$ is a closed 1-form in $\Omega^1_B$ Poincare dual to a fiber surface and $\Omega^1_B$ consists of 1-forms that are bounded near $K$ (see Remark \ref{rmk:bounded} below).

\item The {\it Wrappingness} of $\mu$ with respect to $B$ is
\[Wrp(X,\mu,B) = \underset{\pi \in \Pi}{\text{inf}} \, \, \underset{\theta \in S^1}{\text{min}} \, \, \text{Flux}(X,\mu,\pi^{-1}(\theta)) = \underset{\pi \in \Pi}{\text{inf}} \, \, \underset{\theta \in S^1}{\text{min}} \, \, \underset{\epsilon \rightarrow 0}{\text{lim}} \, \, \frac{1}{\epsilon}\mu \left(\phi_X^{[0,\epsilon]}(\pi^{-1}(\theta)\right)\]

\item The {\it Trunkenness} of $\mu$ with respect to the fibration $\pi$ is
\[Tks(X,\mu,B) = \underset{\pi \in \Pi}{\text{inf}} \, \, \underset{\theta \in S^1}{\text{max}} \, \, \text{Flux}(X,\mu,\pi^{-1}(\theta)) = \underset{\pi \in \Pi}{\text{inf}} \, \, \underset{\theta \in S^1}{\text{max}} \, \, \underset{\epsilon \rightarrow 0}{\text{lim}} \, \, \frac{1}{\epsilon}\mu \left(\phi_X^{[0,\epsilon]}(\pi^{-1}(\theta)\right)\]
\end{enumerate}
\end{definition}

\begin{remark}
\label{rmk:bounded}
The well-definedness of the winding number is addressed in \cite[Section 2.1]{H-SFS}.  In the case where $K \subset Y$ is a collection of periodic orbits of the flow of $X$, the 1-form $\beta$ must be bounded near $K$.  In particular, one can choose tubular polar coordinates $(t,r,\theta)$ near each component of $K$, so that
\[\beta = A dt + B dr + C d\theta\]
The winding number integral is well-defined if the coefficient functions $A,B,C$ are bounded.  Moreover, this condition is independent of the chosen tubular polar coordinates.
\end{remark}

\subsection{Main results}

By adapting the techniques and arguments of Dehornoy-Rechtman, we obtain the following results regarding wrappingness and trunkenness of vector fields with respect to a fibration.  

\begin{theorem}
\label{thrm:homeo-invariance}
    Let $X_1,X_2$ be vector fields on $Y$ that preserve the probability measure $\mu$ and suppose there is a $\mu$-preserving homeomorphism $f$ isotopic to the identity that conjugates the flow of $X_1$ and $X_2$.  Then
    \begin{align*}
        Tks(X_1,\mu,\pi) &= Tks(X_2,\mu,\pi) \\
        Wrp(X_1,\mu,\pi) &= Wrp(X_2,\mu,\pi)
    \end{align*}

\noindent If $B_1,B_2$ are (unions) of periodic orbits of $X_1,X_2$, respectively, and $f(B_1) = B_2$, then
    \begin{align*}
        Tks(X_1,\mu,B_1) &= Tks(X_2,\mu,B_2) \\
        Wrp(X_1,\mu,B_1) &= Wrp(X_2,\mu,B_2)
    \end{align*}
\end{theorem}

\begin{theorem}
\label{thrm:weak-star-convergense}
    Let $(X_n,\mu_n)$ be a sequence of measure-preserving vector fields such that $(X_n,\mu_n)$ converges to $(X,\mu)$ in the weak-$\ast$ sense.  Then we have that
    \begin{align*}
        \lim_{n \rightarrow \infty} Tks(X_n,\mu_n,\pi) & = Tks(X,\mu,\pi) \\
        \lim_{n \rightarrow \infty} Wrp(X_n,\mu_n,\pi) &= Wrp(X,\mu,\pi)
    \end{align*}
    \[\]
\end{theorem}

\subsubsection{Obstruction to surfaces of section}

Just as the trunk of $L$ with respect to $\pi$ is an obstruction to $L$ being braided, the trunkenness of $X$ with respect to $\pi$ is an obstruction to $X$ admitting a surface of section Poincare dual to $[\pi] \in H^1(Y,\ZZ)$. 

\begin{theorem}
Let $X$ be a volume-preserving flow on $Y$ and $\pi: Y \rightarrow S^1$ a fibration.  Suppose that there exists a surface of section representing the Poincare dual to $[\pi]$.  Then for each $X$-invariant measure we have
\[Link(X,\mu,\pi) = Wrp(X,\mu,\pi) = Tks(X,\mu,\pi)\]
\end{theorem}

\begin{proof}
Since there exists a surface of section, we can isotope $\pi$ such that every page of the fibration is positively-transverse to the flow of $X$.  The equalities now immediately follow from the definitions.
\end{proof}

It is interesting to compare this result with the conclusion of Schwartzman-Fried-Sullivan theory of asymptotic cycles (see \cite{H-SFS}), which proves that a surface of section exists if and only if $Link(X,\mu,\pi)$ is positive for every $X$-invariant measure.

\subsubsection{Independence of Helicity}

As with trunkenness for volume-preserving flows in $S^3$, the wrappingness and trunkenness of a volume-preserving flow with respect to a fibered link is independent of helicty.

\begin{theorem}
\label{thrm:helicity}
    Let $Y$ be a 3-manifold and let $(B,\pi)$ be an open book decomposition with planar pages.      There are no functions $f_W$ or $f_T$ such that, for every ergodic volume-preserving vector field on $Y$ with $B$ as a union of periodic orbits, one has
    \begin{align*}
        Wrp(X,\mu,U) &= f_W(Hel(X,\mu),[\beta] \cdot \mu) \\
        Tks(X,\mu,U) &= f_T(Hel(X,\mu),[\beta] \cdot \mu)
    \end{align*}
\end{theorem}

\subsubsection{Existence of periodic orbits}

The trunkenness is defined as the infimum over all smooth fibrations.  If the infimum is actually achieved by some smooth fibration $\pi$, then this forces the existence of a periodic orbit.

\begin{theorem}
\label{thrm:orbit-existence}
Let $X$ be a nonsingular vector field preserving the measure $\mu$.  Suppose that there exists some representative $\pi': Y \rightarrow S^1$ such that
\[Tks(X,\mu,\pi) = \underset{\theta \in S^1}{\text{max}} \, \, Flux(X,\mu,\pi'^{-1}(\theta)) > [\pi] \cdot \mu\]
Then $X$ has a periodic orbit tangent to a fiber of $\pi'$.
\end{theorem}

\begin{corollary}
Let $R$ be an overtwisted Reeb flow on $S^3$.  Suppose that $R$ has an unknotted periodic orbit $U$ and the trunkenness of $R$ with respect to $U$ and the volume form $\alpha \wedge d \alpha$ achieves its infimum.  Then $R$ admits a second unknotted periodic orbit
\end{corollary}

\begin{proof}
    The unknot binds an open book decomposition $(U,\pi)$ with page $D^2$.  This open book decomposition supports the unique tight contact structure on $S^3$.  Therefore, if $\xi$ is overtwisted, no Reeb vector field cannot be transverse to this open book decomposition.  Since the trunkenness of $R$ with respect to $U$ is achieved, then
    \[Link(R,\alpha \wedge d \alpha, \pi) < Tks(R,\alpha \wedge d \alpha,\pi)\]
    Therefore, the Reeb flow admits a periodic orbit tangent to a page of the fibration.  Since the page is a topological disk, this orbit must be unknotted.
\end{proof}

\subsection{Acknowledgements}
I would like to thank Alex Zupan and Nur Saglam for discussions about the trunk of links.

\section{Trunk and Wrapping number}

We can prove the following special cases of Conjecture \ref{conj:fiber-sum}, which will be used in the proof of Theorem \ref{thrm:helicity}.

\begin{proposition}
\label{prop:trunk-embedding}
    Let $L \subset B^3 \subset Y$ be a link embedded in a 3-ball.  Suppose that $\pi: Y \rightarrow S^1$ is a fibration with planar pages.  Then
    \begin{enumerate}
        \item $wind(L,\pi) = wrap(L,\pi) = 0$,
        \item $trunk(L,\pi) = trunk(L,S^3)$,
    \end{enumerate}
    Moreover, if $L' \subset Y \setminus B^3$ is another link in the complement of the 3-ball, then
    \begin{enumerate}
        \item $wrap(L' \cup L,\pi) = wrap(L',\pi)$,
        \item $trunk(L' \cup L,\pi) \geq \text{max}\left( trunk(L',\pi), wrap(L',\pi) + trunk(L,S^3) \right)$
    \end{enumerate}
\end{proposition}

\begin{proof}
    To prove the first statement, we isotope $\partial B^3$ into Roussarie-Thurston general position with respect to the fibration $\pi$.  This implies that $\pi$, restricted to $\partial B^3$, has exactly one local maximum and one local minimum.  Furthermore, by an isotopy we can assume that the image of $\pi|_{B^3}$ is $[0,\epsilon] \in S^1$ for some $\epsilon > 0$.  Therefore, we can assume that $L$ is disjoint from some page of the fibration. Now, let $\pi$ and $\theta$ be a fixed fibration and regular value realizing the wrapping number of $L'$.  Then we can isotopy $B^3$ until its image in $S^1$ is disjoint from $\theta$.  This regular value and fibration realize the wrapping number of $L' \cup L$.

    Now, suppose that $\pi$ is a fixed map to $S^1$ that realizes the trunk of $L$.  Since $B^3$ is simply-connected, then after pulling the fibration back by a covering map $S^1 \rightarrow S^1$ of high degree, we can assume that $\pi(B^3) \subset [0,1/2]$.  In particular, there are embeddings
    \[B^3 \hookrightarrow P \times [0,1/2] \hookrightarrow S^2 \times [0,1/2] \hookrightarrow S^3\]
    that commute with $\pi$.  Here, $P$ is the (abstract) planar page of the fibration, which embeds in $S^2$ by definition.  We can then embed $S^2 \times [0,1/2]$ into $S^3$ so that projection onto $[0,1/2]$ agrees with a Morse function with two critical points.  Consequently,
    \[trunk(L,\pi) \geq trunk(L,S^3)\]
    Conversely, if $f: S^3 \rightarrow \RR$ realizes the trunk of $L$, we can embed $L \subset D^2 \times [0,1]$ into the fibered 3-manifold $Y$ and see that
    \[trunk(L,S^3) \geq trunk(L,Y)\]

    Finally, the trunk satisfies the inequality
    \[trunk(L' \cup L,\pi) \geq wrap(L',\pi) + trunk(L,S^3)\]
    because for any fibration $\pi$, there is some regular level set $\theta$ such that
    \[\#(\pi^{-1}(\theta) \cap L) = trunk(L,S^3)\]
    and this regular level set intersects $L'$ at least $wrap(L',\pi)$-times.  Therefore, this gives a lower bound on the trunk of $L' \cup L$.  Moreover, there exists a regular level set $\theta'$ of this same fibration such that 
    \[\# \pi^{-1}(\theta') \cap L' \geq trunk(L')\]
    This yields the final inequality.    
\end{proof}

\section{Main results}

The proofs of Theorems \ref{thrm:homeo-invariance} and \ref{thrm:weak-star-convergense} are straightforward modifications of Theorem A, Theorem B and Theorem D of \cite{Dehornoy-Rechtman}, except that height functions are replaced by fibrations.

\subsection{Homeomorphism invariance}

\begin{proof}[Proof of Theorem \ref{thrm:homeo-invariance}]
    This proof follows the proof of \cite[Theorem A]{Dehornoy-Rechtman}.  
    Let $f$ be a homeomorphism conjugating the flows of $X_1,X_2$.  Suppose that
    \[\delta = Tks(X_2,\mu,\pi) - Tks(X_1,\mu,\pi) > 0\]
    Let $\pi_n$ be a sequence of fibrations such that
    \[tks(X_1,\mu,\pi_n) := \underset{\theta \in S^1}{\text{max}}\, \lim_{\epsilon \rightarrow 0} \frac{1}{\epsilon} \mu \left(\phi^{[0,\epsilon]}_{X_1}\left(\pi^{-1}_n(\theta) \right) \right)\]
    limits to $Tks(X_1,\mu,\pi)$ as $n$ goes to infinity.  We can smoothly approximate $\pi_n \circ f$ by some fibration $\widetilde{\pi}_n$ such that
    \[\left| \mu(\phi^{[0,\epsilon]}_{X_2}(\widetilde{\pi}_n^{-1}(\theta))) - \mu(\phi^{[0,\epsilon]}_{X_1}(\pi_n^{-1}(\theta))) \right| < \frac{\delta}{4}\]
    for all $\theta \in S^1$ and $\epsilon$ sufficiently small.

    Now choose $n$ sufficiently large that $tks(X_1,\mu,\pi_n) - Tks(X_1,\mu,\pi) < \frac{\delta}{4}$ and let $\theta_n \in S^1$ satisfy
    \[tks(X_2,\mu,\widetilde{\pi}_n) = \text{Flux}(X_2,\mu,\widetilde{\pi}^{-1}_n(\theta_n))\]
    If $\text{Flux}(X_2,\mu,\widetilde{\pi}_n^{-1}(\theta_n)) \geq \text{Flux}(X_1,\mu,\pi_n^{-1}(\theta_n))$, then 
    \[Tks(X_2,\mu,\pi) \leq tks(X_2,\mu,\widetilde{\pi}_n) \leq \frac{\delta}{4} + tks(X_1,\mu,\pi_n) < \frac{\delta}{2} +  Tks(X_1,\mu,\pi)\]
    which implies that $\delta = Tks(X_2,\mu,\pi) - Tks(X_1,\mu,\pi) < \frac{\delta}{2}$, which is a contradiction.
    
    Instead, if $\text{Flux}(X_2,\mu,\widetilde{\pi}_n^{-1}(\theta_n)) < \text{Flux}(X_1,\mu,\pi_n^{-1}(\theta_n))$, then
    \[tks(X_2,\mu,\widetilde{\pi}_n) = \text{Flux}(X_2,\mu,\widetilde{\pi}_n^{-1}(\theta_n)) < \text{Flux}(X_1,\mu,\pi_n^{-1}(\theta_n)) < Tks(X_1,\mu,\pi) + \frac{\delta}{4} < Tks(X_2,\mu,\pi)\]
    which implies the contradiction $tks(X_2,\mu,\widetilde{\pi}_n) < Tks(X_2,\mu,\pi)$.

    Now, suppose that $\delta = Wrp(X_2,\mu,\pi) - Wrp(X_1,\mu,\pi) > 0$.  Let $\pi_n$ be a sequence of fibrations such that
    \[wrp(X_1,\mu,\pi_n) := \underset{\theta \in S^1}{\text{min}} \lim_{\epsilon \rightarrow 0} \frac{1}{\epsilon}\mu \left( \phi^{[0,\epsilon]}_{X_1} \left( \pi^{-1}_n(\theta) \right) \right)\]
    limits to $Wrp(X_1,\mu,\pi)$ as $n$ goes to infinity.  We can smoothly approximate $\pi_n \circ f$ by some fibration $\widetilde{\pi}_n$ such that 
    \[\left| \mu(\phi^{[0,\epsilon]}_{X_2}(\widetilde{\pi}_n^{-1}(\theta))) - \mu(\phi^{[0,\epsilon]}_{X_1}(\pi_n^{-1}(\theta))) \right| < \frac{\delta}{4}\]
    for all $\theta \in S^1$ and $\epsilon$ sufficiently small.  Now choose $n$ sufficiently large that $wrp(X_1,\mu,\pi_n) - Wrp(X_1,\mu,\pi) < \frac{\delta}{4}$ and let $\theta_n \in S^1$ satisfy
    \[ wrp(X_1,\mu,\pi_n) = \text{Flux}(X_1,\mu,\pi^{-1}_n(\theta_n))\]
    Then
    \[\text{Flux}(X_2,\mu,\widetilde{\pi}_n^{-1}(\theta_n)) < wrp(X_1,\mu,\pi_n) + \frac{\delta}{4} < Wrp(X_1,\mu,\pi)  + \frac{\delta}{2} < Wrp(X_2,\mu,\pi)\]
    which is a contradiction.
    
\end{proof}

\subsection{Weak-$\ast$ convergence}

\begin{proof}[Proof of Theorem \ref{thrm:weak-star-convergense}]
This follows the proof of \cite[Theorem B]{Dehornoy-Rechtman}.  As in the proof of Theorem \ref{thrm:homeo-invariance} above, the statement about trunkenness can be proved by the exact same argument, except that height functions are replaced by fibrations.

We will now prove the statement for wrappingness.  Fix $\epsilon > 0$.  Weak-$\ast$ convergence implies that for any surface $S$, if $\delta > 0$ is sufficiently small and $n$ is sufficiently large, then
\[\left| \mu(\phi^{[0,\delta]}_{X}(S)) - \mu(\phi^{[0,\delta]}_{X_n}(S)) \right| < \epsilon \]

Suppose that $Wrp(X_n,\mu_n,\pi)$ does not converge to $Wrp(X,\mu,\pi)$, so that for all $N$ there exists some $n > N$ such that
\[\left| Wrp(X_n,\mu_n,\pi) - Wrp(X,\mu,\pi) \right| > 3 \epsilon\]

First, suppose that $Wrp(X_n,\mu_n,\pi) - Wrp(X,\mu,\pi) > 3 \epsilon$.  Take a sequence of fibrations $\pi_k$ such that
\[\lim_{k \rightarrow 0} wrp(X,\mu,\pi_k) = Wrp(X,\mu,\pi)\]
By extracting a subsequence, we can assume that 
\[0 \leq wrp(X,\mu,\pi_k) - Wrp(X,\mu,\pi) \leq \epsilon\]
for all $k$.  We have
\[Wrp(X_n,\mu,\pi) \leq wrp(X_n,\mu,\pi_k) \leq wrp(X,\mu,\pi_k) + \epsilon \leq Wrp(X,\mu,\pi) + 2 \epsilon < Wrp(X_n,\mu,\pi)\]
which is a contradiction.

Instead, suppose $Wrp(X,\mu,\pi) - Wrp(X_n,\mu_n,\pi) > 3 \epsilon$.  For each $n$, choose a sequence of fibrations $\pi_{n,k}$ such that
\[\lim_{k \rightarrow 0} wrp(X_n,\mu_n,\pi_{n,k}) = Wrp(X_n,\mu_n,\pi)\]
As above, we can assume that for each $n$
\[0 \leq wrp(X_n,\mu_n,\pi_{n,k}) - Wrp(X_n,\mu_n,\pi) \leq \epsilon\]
Then for $k$ sufficiently large,
\[Wrp(X_n,\mu_n,\pi) \geq wrp(X_n,\mu_n,\pi_{n,k}) - \epsilon > wrp(X,\mu,\pi_{n,k}) - 2 \epsilon > Wrp(X,\mu,\pi) - 2 \epsilon > Wrp(X_n,\mu_n,\pi)\]
which is a contradiction.
\end{proof}

\subsection{Existence of periodic orbits}

\begin{proof}[Proof of Theorem \ref{thrm:orbit-existence}]
By assumption, there exists a fibration $\pi': Y \rightarrow S^1$ realizing the trunkenness of the vector field $X$.  Choose $\theta \in S^1$ such that
\[\text{Flux}(X,\mu,(\pi')^{-1}(\theta)) = Tks(X,\mu,\pi)\]
We can divide the page $P_{\theta} = (\pi')^{-1}(\theta)$ into three regions
\[P_{\theta} = P^+_{\theta} \cup P^t_{\theta} \cup P^-_{\theta}\]
according to whether $X$ is positively transverse, tangent, or negatively transverse to the page.  Since the flux is strictly greater than the linking number, both $P^+_{\theta}$ and $P^-_{\theta}$ are nonempty.  Moreover, since they are both open and their intersection is empty, while the page is connected, the set $P^t_{\theta}$ is nonempty as well.

We claim that if $p \in P^t_{\theta}$ is a point whose positive or negative orbit is contained in $P_{\theta}$, then $X$ has a periodic orbit tangent to $P_{\theta}$.  This follows from the generalized Poincare-Bendixson theorem \cite{Schwartz}.  In particular, the $\alpha$- and $\omega$-limit sets of $p$ in $P_{\theta}$ must be either a fixed point, a periodic orbit, or homeomorphic to $T^2$.  The first case cannot occur since $X$ is nonsingular and the third case cannot occur because it would imply that $P_{\theta} = P^t_{\theta}$.

Consequently, given any point $p \in P^t_{\theta}$, its positive and negative orbits leave $P_{\theta}$.  From this point, one can analyze cases and show that if there are no periodic orbits tangent to $P_{\theta}$, one can perturb the fibration $\pi'$ and strictly lower the trunkenness, which violates the assumption that $\pi'$ and $P_{\theta}$ realize the trunkenness of $X$.  The case-by-case arguments in the proof of \cite[Theorem D]{Dehornoy-Rechtman} are completely local, hence carry over immediately to fibrations.
\end{proof}

\section{Independence of Helicity}

In this section, we construct examples of Bott-integrable flows to show that the wrappingness and trunkenness of a volume-preserving flow is independent of its helicity.  The constructions here are inspired by the Bott-integrable fluid flows constructed by Cardona \cite{Cardona} and Bott-integrable Reeb flows constructed by Geiges-Hedicke-Sağlam \cite{GHS-Bott}.  These are constructed from three basic building blocks
\[A = S^1 \times D^2 \qquad B = S^1 \times P \qquad C = T^2 \times [0,1] \cong S^1 \times (S^1 \times[0,1])\]
along embedded tori, where $P$ is a pair of pants surface (i.e. a twice-punctured disk).

Given a volume form $\Omega$ on $Y$, a decomposition of $Y = \cup Y_i$ into a union of basic building blocks, and an $\Omega$-preserving vector field $X$ such that each $\partial Y_i$ is an $X$-invariant torus, we can decompose $\Omega = \sum \Omega_i$ into the sum of $X$-invariant measures, each supported on one component of the decomposition.  In particular, we will construct flows compatible with a decomposition into basic building blocks, then show how to achieve arbitrary wrappingness, trunkenness and helicity by modifying the flow along thickened tori components.

\subsection{Building blocks}

The three basic building blocks we use are
\[A = S^1 \times D^2 \qquad B = S^1 \times P \qquad C = T^2 \times [0,1] = S^1 \times (S^1 \times[0,1])\]
We will construct standard models for volume-preserving flows on each building block, which can then be sewn together to obtain a volume-preserving flow on an entire 3-manifold.

\begin{lemma}
\label{lemma:basic-decomposition}
    There exists a decomposition of basic building blocks
    \[C = B_1 \cup B_2 \cup A_1 \cup A_2\]
    such that $A_2$ is an unknotted solid torus in $C$.
\end{lemma}

\begin{proof}
    First, note that we have a decomposition 
    \[C = T^2 \times [0,1] = S^1 \times (S^1 \times [0,1]) = S^1 \times (P \cup D^2) = S^1 \times P \cup S^1 \times D^2 = B \cup A \]
    Here, the core of the solid torus $A$ can be chosen isotopic to any given simple closed curve on $T^2$.

    We can then decompose
    \[A = C' \cup A_1 = (B_2 \cup A_2) \cup A_1\]
    where the core of $A_2$ represents any simple closed curve on $\partial A$.  In particular, we can assume that it bounds a disk in $A$.  Therefore $A_2$ is unknotted in $C$.
\end{proof}

\subsubsection{Lutz forms}

\begin{definition}
Let $(x_1,x_2,t)$ be coordinates on $T^2 \times [0,1]$.  A $T^2$-invariant 1-form 
\[\alpha = f(t) dx_1 + g(t) dx_2\]
on $T^2 \times [a,b]$ is a {\it Lutz form} if
\[f'g - g' f \neq 0\]
for $t \in [a,b]$.  
\end{definition}
The exterior derivative 
\[d \alpha = -f' dt \wedge dx_1 + g' dt \wedge dx_2\]
of a Lutz form is, by construction, a closed, nonvanishing $T^2$-invariant 2-form.  Let $\Omega = dx_1 \wedge dx_2 \wedge dt$ be a $T^2$-invariant volume form.  Then $d\alpha$ is $\Omega$-dual to the volume-preserving and $T^2$-invariant vector field $X = g' \partial_{x_1} + f' \partial_{x_2}$.  

The following sewing lemma allows us to glue together flows on the building blocks, provided they are defined by Lutz forms near the boundary.

\begin{lemma}[\cite{GHS-Bott}]
    Let $\alpha$ be a Lutz form on $T^2 \times [0,\epsilon] \cup [1-\epsilon,1]$.  There exists an extension of $\alpha$ to a Lutz form on $T^2 \times [0,1]$.
\end{lemma}

\subsubsection{Building block $A = S^1 \times D^2$}.  Choose polar coordinates $(\theta,r,\psi)$ with $\theta,\psi \in [0,2\pi]$ and $r \in [0,R]$.  Define
\[\alpha = \phi(r) d \theta + r^2 d \psi \]
where $\phi > 0$.  Then $\alpha$ is a Lutz form on a neighborhood of the boundary of $A$, provided that $\phi' \neq \frac{2}{r} \phi$.

\subsubsection{Building block $B = S^1 \times P$}.  Let $\theta$ be an angular coordinate on the $S^1$-factor and let $(r_i,\psi_i)$ be coordinates on a collar neighborhood of the $i^{\text{th}}$-boundary component of $P$, with $r_i \in (-1,0]$ and $\psi \in [0,2\pi]$.  

There exists an exact area form $\omega = d \lambda$ on $P$ such that near the boundary, the primitive has the form 
\[\lambda = h_i(r_i) d \psi_i\]
for some function $h_i$ satisfying $h'_i > 0$ \cite[Section 3.2.2]{GHS-Bott}.  We choose the contact form
\[\alpha = \phi d \theta + \lambda\]
where $\phi = 1$ outside a collar neighborhood of $\partial P$ and $\phi(r_i,\psi_i) = \phi_i(r_i)$ for some function satisfying $\phi'_i \geq 0$.  This is a Lutz form near $\partial B$.

\subsubsection{Building block $C = T^2 \times [0,1]$}.  Choose coordinates $x_1,x_2 \in [0,2\pi]$ and $t \in [0,1]$.  Let 
\[X = f(t) \partial_{x_1} + g(t) \partial_{x_2}\]
 be a $T^2$-invariant vector field perserving the volume form $\Omega = dt \wedge dx_1 \wedge dx_2$.  To compute the helicity, define
\[F(t) = \int_0^t f(s) ds \qquad G(t) = \int_0^t g(s) ds\]
Then 
\[\alpha = -F dx_2 + G dx_1\]
that is a Lutz form near the boundary and a primitive for $\iota_X \Omega = -f(t) dt \wedge dx_2 + g(t) dt \wedge dx_1$.  The contribution of the building block $C$ to the global helicity of $d \alpha$ is then
\[Hel(d \alpha) = \int_C (\alpha + \beta) \wedge d \alpha\]
where $\beta$ is a closed 1-form that is cohomologous to $M_1 dx_1 + M_2 dx_2$.  Therefore
\begin{align*}
    \int_C \alpha \wedge d \alpha &= \int_{0}^1 (G f - F g) ds & 
    \int_C \beta \wedge d \alpha &= G(1) \cdot M_2 + F(1) \cdot M_1
\end{align*}

\begin{example}
\label{ex:helicity-sine}
    Take $f(t) = a$ and $g(t) = Q \sin(\pi t) + b$.  Then
    \[F(t) = at \qquad \qquad G(t) = - \frac{Q}{\pi} \cos(\pi t) + bt\]
    and
    \begin{align*}
    \int_0^1 (Gf - fG)dt 
    &= \frac{ab}{2} + Q \int_0^1 \left( t \sin(\pi t) + \frac{1}{\pi}\cos( \pi t) \right) dt \\
    &= \frac{ab}{2} + \frac{Q}{\pi}
    \end{align*}
    Therefore, the helicity contribution is
    \[Hel(d \alpha) = \left( \frac{ab}{2} + \frac{Q}{\pi} \right)+ M_1 \left( \frac{Q}{\pi} + b\right) + a M_2\]
    for some constants $M_1,M_2$.
\end{example}

\subsection{Wrappingness and Trunkenness in basic building blocks}

An embedded, separating 2-torus $T^2 \subset Y$ is {\it unknotted} if there exists a pair $\gamma_1,\gamma_2$ of geometrically dual, embedded curves on $T^2$ that bound disks in the complement of $T^2$.  We can choose coordinates $(x_1,x_2,t)$ on $\nu(T^2) = T^2 \times [-\delta,\delta]$ such that $\gamma_i = \{x_i = \text{const}\}$.  

\begin{proposition}
\label{prop:trunkenness-unknotted-torus}
Let $T^2 \times [0,1]$ be an unknotted thickened torus in $Y$.  Consider the $T^2$-invariant flow generated by $X = f(t) \partial_{x_1} + g(t) \partial_{x_2}$, which preserves the volume-form $\Omega = dx_1 \wedge dx_2 \wedge dy$.  Then
\[Tks(X,\Omega,p) = 4\pi \int_{0}^{1} \text{min}(|f|,|g|) dt\]
\end{proposition}

\begin{proof}
    The flowlines of $X$ foliate the invariant tori of $T^2 \times [0,1]$.  If $f/g$ is in $\mathbb{Q} \cup \{\infty\}$, then these flowlines are torus knots, since the torus is assumed to be unknotted.  The trunk of the $(p,q)$-torus knot is $2 \, \text{min}(|p|,|q|)$ \cite{Zupan}.  Therefore, if $X = p \partial_{x_1} + q \partial_{x_2}$ with $|p| \leq |q|$ and $p,q$ relatively prime, the annuli $A = \{x_1 = 0\} \cup \{x_1 = \pi\}$ extend to an embedded disk in $Y$ realizing the minimal trunk of the $(p,q)$-torus knot.  If $f = p$ and $g = q$ are locally constant, the local contribution to the trunkenness of $X$ is given by integrating $|\iota_X \Omega| = |f| dx_2 \wedge dt - |g| dx_1 \wedge dt$ over a neighborhood in $A$.
    
    Trunkenness is an order-1 invariant, as scaling the vector field by a constant $\lambda$ scales the trunkenness by $\lambda$.  Moreover, by Theorem \ref{thrm:weak-star-convergense}, we can compute the trunkenness of irrational slopes as the limit of the trunkenness of rational slopes.  Consequently, letting $f,g$ vary continuously in $t$, we obtain the required integral formula for trunkenness in a thickened, unknotted torus.
\end{proof}

\begin{proposition}
\label{prop:wrappingness-thickened-torus}
    Let $T^2 \times [0,1] \subset Y$ be embedded such that the fibration restricts to the projection map
    \[ p: S^1 \times (S^1 \times [0,1]) \rightarrow S^1\]
    and Let $p$ be the fibration defined by the closed 1-form $dx_1$.  Consider the $T^2$-invariant vector field $X = f(t) \partial_{x_1} + g(t) \partial_{x_2}$.  Then
    \[Wrp(X,\Omega,p) = 4 \pi \int_0^1 |f| dt\]
\end{proposition}

\begin{proof}
    The wrapping number of a braided link is equal to the absolute value of its winding number.  If $\frac{f}{g} = \frac{rp}{rq}$ is rational, with $r > 0$ some real number and $p,q$ integers, then the flowlines are $(p,q)$-curves on $T^2$, which are braided with respect to the fibration.  The wrapping number of such a curve is therefore $|p|$ and the wrappingness of the ergodic measure concentrated on this periodic orbit is $r|p|$.  As in the previous proposition, we can approximate irrational slopes by rational slopes using \ref{thrm:weak-star-convergense} and then integrate.
\end{proof}

Combining these two propositions with Example \ref{ex:helicity-sine}, it is clear one can modify the helicity independently of the trunkenness and wrappingness.

\begin{proof}[Proof of Theorem \ref{thrm:helicity}]
Let $X$ be a volume-preserving flow on $Y$ and suppose that $C = T^2 \times [0,1]$ can be embedded in $Y$ such that $X$ is $T^2$-invariant and $\pi$ restricts to a fibration on $C$.  Then by Lemma \ref{lemma:basic-decomposition}, we can decompose $C$ into basic building blocks and find an unknotted solid torus compatible with a decomposition into basic building blocks.  Let $C_1$ be a neighborhood of the boundary of this solid torus.  We can modify $X$ on $C$, fixed near the boundary, such that on $C_1$ we have $X$ as in Example \ref{ex:helicity-sine}.  Provided that $|a| < |b| - |Q|$, the helicity of $X$ varies as $Q$ varies but the linking nunmber and trunkenness contribution is fixed, according to Proposition \ref{prop:trunkenness-unknotted-torus}.  Similarly, if we take a parallel copy $C_2$ of $C$, we can again use Example \ref{ex:helicity-sine} combined with Proposition \ref{prop:wrappingness-thickened-torus}.  The wrappingness is determined by $|a|$ but the helicity varies in $Q$.
\end{proof}

Dehornoy and Rechtman futher remark in the proof of \cite[Theorem C]{Dehornoy-Rechtman} that by a theorem of Katok \cite{Katok}, volume-preserving flows can be $C^1$-perturbed to ergodic flows.

\printbibliography[title={Bibliography}]

@article {Katok,
    AUTHOR = {Katok, A. B.},
     TITLE = {Ergodic perturbations of degenerate integrable {H}amiltonian
              systems},
   JOURNAL = {Izv. Akad. Nauk SSSR Ser. Mat.},
  FJOURNAL = {Izvestiya Akademii Nauk SSSR. Seriya Matematicheskaya},
    VOLUME = {37},
      YEAR = {1973},
     PAGES = {539--576},
      ISSN = {0373-2436},
   MRCLASS = {58F05 (28A65)},
  MRNUMBER = {331425},
MRREVIEWER = {K.\ Krzy\.{z}ewski},
}

@article {Dehornoy-Rechtman,
    AUTHOR = {Rechtman, Ana and Dehornoy, Pierre},
     TITLE = {The trunkenness of a volume-preserving vector field},
   JOURNAL = {Nonlinearity},
  FJOURNAL = {Nonlinearity},
    VOLUME = {30},
      YEAR = {2017},
    NUMBER = {11},
     PAGES = {4089--4110},
      ISSN = {0951-7715,1361-6544},
   MRCLASS = {37C15 (34C40 37C27 76W05)},
  MRNUMBER = {3718732},
MRREVIEWER = {Alejandro\ Mario\ Mes\'{o}n},
       DOI = {10.1088/1361-6544/aa83a8},
       URL = {https://doi.org/10.1088/1361-6544/aa83a8},
}

@article {Ozawa,
    AUTHOR = {Ozawa, Makoto},
     TITLE = {Waist and trunk of knots},
   JOURNAL = {Geom. Dedicata},
  FJOURNAL = {Geometriae Dedicata},
    VOLUME = {149},
      YEAR = {2010},
     PAGES = {85--94},
      ISSN = {0046-5755,1572-9168},
   MRCLASS = {57M25 (57M27)},
  MRNUMBER = {2737680},
MRREVIEWER = {Stefan\ K.\ Friedl},
       DOI = {10.1007/s10711-010-9466-y},
       URL = {https://doi.org/10.1007/s10711-010-9466-y},
}

@article {Schwartz,
    AUTHOR = {Schwartz, Arthur J.},
     TITLE = {A generalization of a {P}oincar\'{e}-{B}endixson theorem to
              closed two-dimensional manifolds},
   JOURNAL = {Amer. J. Math.},
  FJOURNAL = {American Journal of Mathematics},
    VOLUME = {85},
      YEAR = {1963},
     PAGES = {453--458; errata: {\bf 85 (1963), 753}},
      ISSN = {0002-9327,1080-6377},
   MRCLASS = {34.65 (57.48)},
  MRNUMBER = {155061},
MRREVIEWER = {Bruce\ L.\ Reinhart},
       URL =
              {http://links.jstor.org/sici?sici=0002-9327(196307)85:3<453:AGOAPT>2.0.CO;2-8&origin=MSN},
}

@article {Cardona,
    AUTHOR = {Cardona, Robert},
     TITLE = {The topology of {B}ott integrable fluids},
   JOURNAL = {Discrete Contin. Dyn. Syst.},
  FJOURNAL = {Discrete and Continuous Dynamical Systems. Series A},
    VOLUME = {42},
      YEAR = {2022},
    NUMBER = {9},
     PAGES = {4321--4345},
      ISSN = {1078-0947,1553-5231},
   MRCLASS = {35Q31 (37C86 37J35 76B03)},
  MRNUMBER = {4455234},
       DOI = {10.3934/dcds.2022054},
       URL = {https://doi.org/10.3934/dcds.2022054},
}

@article{GHS-Bott,
author = {Geiges, Hansjörg and Hedicke, Jakob and Sağlam, Murat},
title = {Bott-integrable Reeb flows on 3-manifolds},
journal = {Journal of the London Mathematical Society},
volume = {109},
number = {1},
pages = {e12859},
doi = {https://doi.org/10.1112/jlms.12859},
url = {https://londmathsoc.onlinelibrary.wiley.com/doi/abs/10.1112/jlms.12859},
eprint = {https://londmathsoc.onlinelibrary.wiley.com/doi/pdf/10.1112/jlms.12859},
year = {2024}
}

@article {H-SFS,
    AUTHOR = {Hryniewicz, Umberto L.},
     TITLE = {A note on {S}chwartzman-{F}ried-{S}ullivan theory, with an
              application},
   JOURNAL = {J. Fixed Point Theory Appl.},
  FJOURNAL = {Journal of Fixed Point Theory and Applications},
    VOLUME = {22},
      YEAR = {2020},
    NUMBER = {1},
     PAGES = {Paper No. 25, 20},
      ISSN = {1661-7738,1661-7746},
   MRCLASS = {37D40 (37C40 37E30)},
  MRNUMBER = {4067303},
MRREVIEWER = {Norikazu\ Hashiguchi},
       DOI = {10.1007/s11784-020-0757-0},
       URL = {https://doi.org/10.1007/s11784-020-0757-0},
}

@article {Zupan,
    AUTHOR = {Zupan, Alexander},
     TITLE = {A lower bound on the width of satellite knots},
   JOURNAL = {Topology Proc.},
  FJOURNAL = {Topology Proceedings},
    VOLUME = {40},
      YEAR = {2012},
     PAGES = {179--188},
      ISSN = {0146-4124,2331-1290},
   MRCLASS = {57M25 (57M27)},
  MRNUMBER = {2817298},
MRREVIEWER = {Jennifer\ Schultens},
       DOI = {10.1016/j.physletb.2012.07.017},
       URL = {https://doi.org/10.1016/j.physletb.2012.07.017},
}

\end{document}